\documentclass{amsart}
\usepackage{amsmath, amssymb, amsthm, epsfig}
\usepackage{amssymb, latexsym}
\usepackage{url}

\DeclareMathOperator{\Tor}{\mathrm{Tor}}

\begin{document}
\bibliographystyle{plain}

\newtheorem{theorem}{Theorem}[section]
\newtheorem{lemma}{Lemma}[section]
\newtheorem{corollary}{Corollary}[section]
\newtheorem{conjecture}{Conjecture}
\newtheorem{definition}{Definition}
 
\newcommand{\mc}{\mathcal}
\newcommand{\A}{\mc A}
\newcommand{\B}{\mc B}
\newcommand{\cc}{\mc C}
\newcommand{\D}{\mc D}
\newcommand{\E}{\mc E}
\newcommand{\F}{\mc F}
\newcommand{\G}{\mc G}
\newcommand{\hH}{\mc H}
\newcommand{\I}{\mc I}
\newcommand{\J}{\mc J}
\newcommand{\K}{\mc K}
\newcommand{\eL}{\mc L}
\newcommand{\M}{\mc M}
\newcommand{\eN}{\mc N}
\newcommand{\pp}{\mc P}
\newcommand{\qq}{\mc Q}
\newcommand{\rr}{\mc R}
\newcommand{\Ss}{\mc S}
\newcommand{\U}{\mc U}
\newcommand{\V}{\mc V}
\newcommand{\W}{\mc W}
\newcommand{\X}{\mc X}
\newcommand{\Y}{\mc Y}
\newcommand{\zZ}{\mc Z}
\newcommand{\C}{\mathbb{C}}
\newcommand{\R}{\mathbb{R}}
\newcommand{\Q}{\mathbb{Q}}
\newcommand{\T}{\mathbb{T}}
\newcommand{\Z}{\mathbb{Z}}
\newcommand{\aA}{\mathfrak A}
\newcommand{\bB}{\mathfrak B}
\newcommand{\cC}{\mathfrak C}
\newcommand{\dD}{\mathfrak D}
\newcommand{\ee}{\mathfrak E}
\newcommand{\ff}{\mathfrak F}
\newcommand{\iI}{\mathfrak I}
\newcommand{\mM}{\mathfrak M}
\newcommand{\nN}{\mathfrak N}
\newcommand{\pP}{\mathfrak P}
\newcommand{\sS}{\mathfrak S}
\newcommand{\uU}{\mathfrak U}
\newcommand{\fb}{f_{\beta}}
\newcommand{\fg}{f_{\gamma}}
\newcommand{\gb}{g_{\beta}}
\newcommand{\vphi}{\varphi}
\newcommand{\p}{\varphi}
\newcommand{\vep}{\varepsilon}
\newcommand{\bo}{\boldsymbol 0}
\newcommand{\bone}{\boldsymbol 1}
\newcommand{\balpha}{\boldsymbol \alpha}
\newcommand{\bbeta}{\boldsymbol \beta}
\newcommand{\ba}{\boldsymbol a}
\newcommand{\bb}{\boldsymbol b}
\newcommand{\bc}{\boldsymbol c}
\newcommand{\be}{\boldsymbol e}
\newcommand{\bff}{\boldsymbol f}
\newcommand{\bk}{\boldsymbol k}
\newcommand{\bell}{\boldsymbol \ell}
\newcommand{\bm}{\boldsymbol m}
\newcommand{\bn}{\boldsymbol n}
\newcommand{\bgamma}{\boldsymbol \gamma}
\newcommand{\blambda}{\boldsymbol \lambda}
\newcommand{\btheta}{\boldsymbol \theta}
\newcommand{\bq}{\boldsymbol q}
\newcommand{\bt}{\boldsymbol t}
\newcommand{\bu}{\boldsymbol u}
\newcommand{\bv}{\boldsymbol v}
\newcommand{\bw}{\boldsymbol w}
\newcommand{\bx}{\boldsymbol x}
\newcommand{\bwy}{\boldsymbol y}
\newcommand{\bnu}{\boldsymbol \nu}
\newcommand{\bxi}{\boldsymbol \xi}
\newcommand{\bz}{\boldsymbol z}
\newcommand{\tF}{\widehat F}
\newcommand{\tG}{\widehat{G}}
\newcommand{\oK}{\overline{K}}
\newcommand{\oKt}{\overline{K}^{\times}}
\newcommand{\oQ}{\overline{\Q}}
\newcommand{\oq}{\oQ^{\times}}
\newcommand{\oQt}{\oQ^{\times}/\Tor\bigl(\oQ^{\times}\bigr)}
\newcommand{\ot}{\Tor\bigl(\oQ^{\times}\bigr)}
\newcommand{\h}{\frac12}
\newcommand{\hh}{\tfrac12}
\newcommand{\dt}{\text{\rm d}t}
\newcommand{\dx}{\text{\rm d}x}
\newcommand{\dbx}{\text{\rm d}\bx}
\newcommand{\dy}{\text{\rm d}y}
\newcommand{\dmu}{\text{\rm d}\mu}
\newcommand{\dnu}{\text{\rm d}\nu}
\newcommand{\dla}{\text{\rm d}\lambda}
\newcommand{\dlav}{\text{\rm d}\lambda_v}
\newcommand{\trho}{\widetilde{\rho}}
\newcommand{\dtrho}{\text{\rm d}\widetilde{\rho}}
\newcommand{\drho}{\text{\rm d}\rho}
\def\today{\number\time, \ifcase\month\or
January\or February\or March\or April\or May\or June\or
July\or August\or September\or October\or November\or December\fi
\space\number\day, \number\year}
  
\title[Mahler measure]{Lower bounds for Mahler measure that\\depend on the number of monomials}
\author{Shabnam Akhtari and Jeffrey~D.~Vaaler}
\thanks{}
\medskip

\address{Department of Mathematics, University of Oregon, Eugene, Oregon 97402 USA \\
Max Planck Institute for Mathematics,
Vivatsgasse 7,
53111 Bonn,
Germany}
\email{akhtari@uoregon.edu}
\medskip

\address{Department of Mathematics, University of Texas, Austin, Texas 78712 USA}
\email{vaaler@math.utexas.edu}

\thanks{Shabnam Akhtari's research is funded by the NSF grant DMS-1601837.}

\begin{abstract}  We prove a new lower bound for the Mahler measure of a polynomial in one and in several variables
that depends on the complex coefficients, and the number of monomials.  In one variable our result generalizes
a classical inequality of Mahler.  In $M$ variables our result depends on $\Z^M$ as an ordered group, and in
general our lower bound depends on the choice of ordering.   
\end{abstract}
\maketitle
\numberwithin{equation}{section}

\section{Introduction}

Let $P(z)$ be a polynomial in $\C[z]$ that is not identically zero.  We assume to begin with that $P$ has degree $N$,
and that $P$ factors into linear factors in $\C[z]$ as
\begin{equation}\label{intro1}
P(z) = c_0 + c_1z + c_2 z^2 + \cdots + c_N z^N = c_N \prod_{n = 1}^N (z - \alpha_n).
\end{equation}
If $e : \R/\Z \rightarrow \T$ is the continuous isomorphism given by $e(t) = e^{2 \pi i t}$, then the Mahler measure of 
$P$ is the positive real number
\begin{equation}\label{intro5}
\mM(P) = \exp\biggl(\int_{\R/\Z} \log \bigl|P\bigl(e(t)\bigr)\bigr|\ \dt\biggr) = |c_N| \prod_{n = 1}^N \max\{1, |\alpha_n|\}.
\end{equation}
The equality on the right of (\ref{intro5}) follows from Jensen's formula.  If $P_1(z)$ and $P_2(z)$ are both nonzero
polynomials in $\C[z]$, then it is immediate from (\ref{intro5}) that
\begin{equation*}\label{intro7}
\mM\bigl(P_1 P_2\bigr) = \mM\big(P_1\bigr) \mM\bigl(P_2\bigr).
\end{equation*}
 Mahler measure plays an important role in number theory and in algebraic dynamics, as discussed in \cite{everest1999}, 
\cite{pritsker2007}, \cite[Chapter 5]{schmidt1995}, and \cite{smyth2007}.  Here we restrict our attention to the 
problem of proving a lower bound for $\mM(P)$ when the polynomial $P(z)$ has complex coefficients. We establish an
analogous result for polynomials in several variables.

For $P(z)$ of degree $N$ and given by (\ref{intro1}), there is a well known lower bound due to Mahler which asserts that
\begin{equation}\label{intro10}
|c_n| \le \binom{N}{n} \mM(P),\quad\text{for each $n = 0, 1, 2, \dots , N$}.
\end{equation}
The inequality (\ref{intro10}) is implicit in \cite{mahler1960}, and is stated explicitly in \cite[section 2]{mahler1962}, (see also
the proof in \cite[Theorem 1.6.7]{bombieri2006}).  If
\begin{equation*}\label{intro15}
P(z) = (z \pm 1)^N,
\end{equation*}
then there is equality in (\ref{intro10}) for each $n = 0, 1, 2, \dots , N$.    

We now assume that $P(z)$ is a polynomial in $\C[z]$ that is not identically zero, and we assume that $P(z)$ is given by
\begin{equation}\label{intro20}
P(z) = c_0 z^{m_0} + c_1 z^{m_1} + c_2 z^{m_2} + \cdots + c_N z^{m_N},
\end{equation}
where $N$ is a nonnegative integer, and $m_0, m_1, m_2, \dots , m_N$, are nonnegative integers such that
\begin{equation}\label{intro25}
m_0 < m_1 < m_2 < \cdots < m_N.
\end{equation}
We wish to establish a lower bound for $\mM(P)$ which depends on the coefficients and on the number of monomials,
but which does {\it not} depend on the degree of $P$.  Such a result was recently proved by Dobrowolski and 
Smyth \cite{smyth2017}.  We use a similar argument, but we obtain a sharper result that includes
Mahler's inequality (\ref{intro10}) as a special case.

\begin{theorem}\label{thmintro1}  Let $P(z)$ be a polynomial in $\C[z]$ that is not identically zero, and is given by
{\rm (\ref{intro20})}.  Then we have
\begin{equation}\label{intro30}
|c_n| \le \binom{N}{n} \mM(P),\quad\text{for each $n = 0, 1, 2, \dots , N$}.
\end{equation}
\end{theorem}

Let $f : \R/\Z \rightarrow \C$ be a trigonometric polynomial, not identically zero, and a sum of at most $N + 1$ distinct 
characters.  Then we can write $f$ as
\begin{equation}\label{intro33}
f(t) = \sum_{n = 0}^N c_n e(m_n t),
\end{equation}
where $c_0, c_1, c_2, \dots , c_N$, are complex coefficients, and $m_0, m_1, m_2, \dots , m_N$, are integers 
such that
\begin{equation*}\label{intro35}
m_0 < m_1 < m_2 < \cdots < m_N.
\end{equation*}
As $f$ is not identically zero, the Mahler measure of $f$ is the positive number
\begin{equation*}\label{intro38}
\mM(f) = \exp\biggl(\int_{\R/\Z} \log |f(t)|\ \dt\biggr).
\end{equation*}
It is trivial that $f(t)$ and $e(-m_0 t)f(t)$ have the same Mahler measure.  Thus we get the following alternative formulation
of Theorem \ref{thmintro1}.

\begin{corollary}\label{corintro1}  Let $f(t)$ be a trigonometric polynomial with complex coefficients that is not identically
zero, and is given by {\rm (\ref{intro33})}.  Then we have
\begin{equation}\label{intro40}
|c_n| \le \binom{N}{n} \mM(f),\quad\text{for each $n = 0, 1, 2, \dots , N$}.
\end{equation}
\end{corollary}

For positive integers $M$ we will prove an extension of Corollary \ref{corintro1} to trigonometric polynomials
\begin{equation}\label{intro50}
F : (\R/\Z)^M \rightarrow \C,
\end{equation}
that are not identically zero.  The Fourier transform of $F$ is the function
\begin{equation*}\label{intro53}
\tF : \Z^M \rightarrow \C,
\end{equation*}
defined at each lattice point $\bk$ in $\Z^M$ by
\begin{equation}\label{intro55}
\tF(\bk) = \int_{(\R/\Z)^M} F(\bx) e\bigl(-\bk^T \bx\bigr)\ \dbx.
\end{equation}
In the integral on the right of (\ref{intro55}) we write $\dbx$ for integration with respect to a Haar measure on the
Borel subsets of $(\R/\Z)^M$ normalized so that $(\R/\Z)^M$ has measure $1$.  We write $\bk$ for a (column) vector
in $\Z^M$, $\bk^T$ for the transpose of $\bk$, $\bx$ for a (column) vector in $(\R/\Z)^M$, and therefore
\begin{equation*}\label{intro60}
\bk^T \bx = k_1 x_1 + k_2 x_2 + \cdots + k_N x_N.
\end{equation*} 
As $F$ is not identically zero, the Mahler measure of $F$ is the positive real number
\begin{equation*}\label{intro75}
\mM(F) = \exp\biggl(\int_{(\R/\Z)^M} \log \bigl|F(\bx)\bigr|\ \dbx\biggr).
\end{equation*}
We assume that $\sS \subseteq \Z^M$ is a nonempty, finite set that contains the support of $\tF$.  That is, we assume that
\begin{equation}\label{intro80}
\{\bk \in \Z^M : \tF(\bk) \not= 0\} \subseteq \sS,
\end{equation}
and therefore $F$ has the representation
\begin{equation}\label{intro70}
F(\bx) = \sum_{\bk \in \sS} \tF(\bk) e\bigl(\bk^T \bx\bigr).
\end{equation}
Basic results in this setting can be found in Rudin \cite[Sections 8.3 and 8.4]{rudin1962}.  

If $\balpha = (\alpha_m)$ is a (column) vector in $\R^M$, we write
\begin{equation*}\label{intro100}
\vphi_{\balpha} : \Z^M \rightarrow \R
\end{equation*}
for the homomorphism given by
\begin{equation}\label{intro105}
\vphi_{\balpha}(\bk) = \bk^T \balpha = k_1 \alpha_1 + k_2 \alpha_2 + \cdots + k_M \alpha_M.
\end{equation}
It is easy to verify that $\vphi_{\balpha}$ is an injective homomorphism if and only if the coordinates
$\alpha_1, \alpha_2, \dots , \alpha_M$, are $\Q$-linearly independent real numbers. 

Let the nonempty, finite set $\sS \subseteq \Z^M$ have cardinality $N+1$,
where $0 \le N$.  If $\vphi_{\balpha}$ is an injective homomorphism, then the set
\begin{equation*}\label{intro110}
\big\{\vphi_{\balpha}(\bk) : \bk \in \sS\big\}
\end{equation*}
consists of exactly $N+1$ real numbers.  It follows that the set $\sS$ can be indexed so that
\begin{equation}\label{intro115}
\sS = \big\{\bk_0, \bk_1, \bk_2, \dots , \bk_N\big\},
\end{equation}
and
\begin{equation}\label{intro120}
\vphi_{\balpha}\bigl(\bk_0\bigr) < \vphi_{\balpha}\bigl(\bk_1\bigr) 
	< \vphi_{\balpha}\bigl(\bk_2\bigr) < \cdots < \vphi_{\balpha}\bigl(\bk_N\bigr).
\end{equation}
By using a limiting argument introduced in a paper of Boyd \cite{boyd1981}, we will prove the following
generalization of (\ref{intro40}).

\begin{theorem}\label{thmintro2}  Let $F : (\R/\Z)^M \rightarrow \C$ be a trigonometric polynomial that is not
identically zero, and is given by {\rm (\ref{intro70})}.  Let $\vphi_{\balpha} : \Z^M \rightarrow \R$ be an
injective homomorphism, and assume that the finite set $\sS$, which contains the support of $\tF$, is indexed so 
that {\rm (\ref{intro115})} and {\rm (\ref{intro120})} hold.  Then we have
\begin{equation}\label{intro125}
\bigl|\tF\bigl(\bk_n\bigr)\bigr| \le \binom{N}{n} \mM(F),\quad\text{for each $n = 0, 1, 2, \dots , N$}.
\end{equation}
\end{theorem}

Let $F$ and $\vphi_{\balpha} : \Z^M \rightarrow \R$ be as in the statement of Theorem \ref{thmintro2}, and then let 
$\vphi_{\bbeta} : \Z^M \rightarrow \R$ be a second injective homomorphism.  It follows that $\sS$ can be indexed so that
(\ref{intro115}) and (\ref{intro120}) hold, and $\sS$ can also be indexed so that
\begin{equation}\label{intro145}
\sS = \big\{\bell_0, \bell_1, \bell_2, \dots , \bell_N\big\},
\end{equation}
and
\begin{equation}\label{intro150}
\vphi_{\bbeta}\bigl(\bell_0\bigr) < \vphi_{\bbeta}\bigl(\bell_1\bigr) 
	< \vphi_{\bbeta}\bigl(\bell_2\bigr) < \cdots < \vphi_{\bbeta}\bigl(\bell_N\bigr).
\end{equation}
In general the indexing (\ref{intro115}) is distinct from the indexing (\ref{intro145}).  Therefore the system of inequalities
\begin{equation}\label{intro155}
\bigl|\tF\bigl(\bk_n\bigr)\bigr| \le \binom{N}{n} \mM(F),\quad\text{for each $n = 0, 1, 2, \dots , N$},
\end{equation}
and
\begin{equation}\label{intro160}
\bigl|\tF\bigl(\bell_n\bigr)\bigr| \le \binom{N}{n} \mM(F),\quad\text{for each $n = 0, 1, 2, \dots , N$},
\end{equation}
which follow from Theorem \ref{thmintro2}, are different, and in general neither system of inequalities implies the other.

\section{Proof of Theorem \ref{thmintro1}}

It follows from (\ref{intro5}) that the polynomial $P(z)$, and the polynomial $z^{-m_0} P(z)$, have the same
Mahler measure.  Hence we may assume without loss of generality that the exponents $m_0, m_1, m_2, \dots , m_N$, in the representation (\ref{intro20}) satisfy the more restrictive condition
\begin{equation}\label{mahler250}
0 = m_0 < m_1 < m_2 < \cdots < m_N.
\end{equation}

If $N = 0$ then (\ref{intro30}) is trivial.  If $N = 1$, then
\begin{equation*}\label{mahler252}
\binom{1}{0} = \binom{1}{1} = 1,
\end{equation*}
and using Jensen's formula we find that
\begin{equation*}\label{mahler254}
\mM\bigl(c_0 + c_1 z^{m_1}\bigr) = \max\{|c_0|, |c_1|\}.
\end{equation*}
Therefore the inequality (\ref{intro30}) holds if $N = 1$.  Throughout the remainder of the proof we assume that 
$2 \le N$, and we argue by induction on $N$.  Thus we assume that the inequality (\ref{intro30}) holds for polynomials 
that can be expressed as a sum of strictly less than $N + 1$ monomials.

Besides the polynomial
\begin{equation}\label{mahler256}
P(z) = c_0 z^{m_0} + c_1 z^{m_1} + c_2 z^{m_2} + \cdots + c_N z^{m_N},
\end{equation}
we will work with the polynomial
\begin{equation}\label{mahler258}
Q(z) = z^{m_N} P\bigl(z^{-1}\bigr) = c_0 z^{m_N - m_0} + c_1 z^{m_N - m_1} + c_2 z^{m_N- m_2} + \cdots + c_N.
\end{equation}
It follows from (\ref{intro5}) that
\begin{equation}\label{mahler260}
\mM(Q) = \exp\biggl(\int_{\R/\Z} \log \bigl|e(m_N t) P\bigl(e(- t)\bigr)\bigr|\ \dt\biggr) = \mM(P).
\end{equation}
Next we apply an inequality of Mahler \cite{mahler1961} to conclude that both
\begin{equation}\label{mahler262}
\mM\bigl(P^{\prime}\bigr) \le m_N \mM(P),\quad\text{and}\quad \mM\bigl(Q^{\prime}\bigr) \le m_N \mM(Q).
\end{equation} 
Because
\begin{equation*}\label{mahler266}
P^{\prime}(z) = \sum_{n = 1}^N c_n m_n z^{m_n - 1}
\end{equation*}
is a sum of strictly less than $N + 1$ monomials, we can apply the inductive hypothesis to $P^{\prime}$.  It follows that
\begin{equation}\label{mahler271}
|c_n| m_n \le \binom{N-1}{n-1} \mM\bigl(P^{\prime}\bigr) \le m_N \binom{N-1}{n-1} \mM(P)
\end{equation}  
for each $n = 1, 2, \dots , N$.  As 
\begin{equation*}\label{276}
m_0 = 0,\quad\text{and}\quad \binom{N-1}{-1} = 0,
\end{equation*}
it is trivial that (\ref{mahler271}) also holds at $n = 0$.

In a similar manner, 
\begin{equation*}\label{mahler281}
Q^{\prime}(z) = \sum_{n = 0}^{N-1} c_n (m_N - m_n) z^{m_N - m_n - 1}
\end{equation*}
is a sum of strictly less that $N + 1$ monomials.  We apply the inductive hypothesis to $Q^{\prime}$, and get the inequality
\begin{equation}\label{mahler286}
|c_n| (m_N - m_n) \le \binom{N-1}{N - 1 - n} \mM\bigl(Q^{\prime}\bigr) \le m_N \binom{N-1}{n} \mM(Q)
\end{equation}
for each $n = 0, 1, 2, \dots , N-1$.  In this case we have
\begin{equation*}\label{mahler291}
(m_N - m_N) = 0,\quad\text{and}\quad \binom{N-1}{N} = 0,
\end{equation*}
and therefore (\ref{mahler286}) also holds at $n = N$.

To complete the proof we use the identity (\ref{mahler260}), and we apply the inequality (\ref{mahler271}), 
and the inequality (\ref{mahler286}).  In this way we obtain the bound
\begin{equation}\label{mahler296}
\begin{split}
|c_n| m_N &= |c_n| m_n + |c_n| (m_N - m_n)\\
                 &\le m_N \binom{N-1}{n-1} \mM(P) + m_N \binom{N-1}{n} \mM(P)\\
                 &= m_N \binom{N}{n} \mM(P).
\end{split}
\end{equation}
This verifies (\ref{intro30}).

\section{Archimedean orderings in the group $\Z^M$}

In this section we consider $\Z^M$ as an ordered group.  To avoid degenerate situations, we assume throughout 
this section that $2 \le M$.

Let $\balpha$ belong to $\R^M$, and let $\vphi_{\balpha} : \Z^M \rightarrow \R$ be the homomorphism defined by
(\ref{intro105}).  We assume that the coordinates $\alpha_1, \alpha_2, \dots , \alpha_M$, are $\Q$-linearly independent 
so that $\vphi_{\balpha}$ is an injective homomorphism.  It follows, as in \cite[Theorem 8.1.2 (c)]{rudin1962}, that 
$\vphi_{\balpha}$ induces an archimedean ordering in the group $\Z^M$.  That is, if $\bk$ and $\bell$ are distinct points 
in $\Z^M$ we write $\bk < \bell$ if and only if 
\begin{equation*}\label{order-5}
\vphi_{\balpha}(\bk) = \bk^T \balpha < \vphi_{\balpha}(\bell) = \bell^T \balpha
\end{equation*}
in $\R$.  Therefore $\bigl(\Z^M, <\bigr)$ is an ordered group, and the order is archimedean.  If $\sS \subseteq \Z^M$ is 
a nonempty, finite subset of cardinality $N + 1$, then the elements of $\sS$ can be indexed so that
\begin{equation}\label{order1}
\sS = \big\{\bk_0, \bk_1, \bk_2, \dots , \bk_N\big\}
\end{equation}
and
\begin{equation}\label{order5}
\bk_0^T \balpha < \bk_1^T \balpha < \bk_2^T \balpha < \cdots < \bk_N^T \balpha.
\end{equation}
A more general discussion of ordered groups is given in \cite[Chapter 8]{rudin1962}.  Here we require only the indexing 
(\ref{order1}) that is induced in the finite subset $\sS$ by the injective homomorphism $\vphi_{\balpha}$.

If $\bb = (b_m)$ is a (column) vector in $\Z^M$, we define the norm
\begin{equation}\label{order7}
\|\bb\|_{\infty} = \max\big\{|b_m| : 1 \le m \le M\big\}.
\end{equation} 
And if $\sS \subseteq \Z^M$ is a nonempty, finite subset we write
\begin{equation*}\label{order10}
\|\sS\|_{\infty} = \max\big\{\|\bk\|_{\infty} : \bk \in \sS\big\}.
\end{equation*}
Following Boyd \cite{boyd1981}, we define the function
\begin{equation*}\label{order20}
\nu : \Z^M \setminus \{\bo\} \rightarrow \{1, 2, 3, \dots \}
\end{equation*}
by
\begin{equation}\label{order25}
\nu(\ba) = \min\big\{\|\bb\|_{\infty} : \text{$\bb \in \Z^M$, $\bb \not= \bo$, and $\bb^T \ba = 0$}\big\}.
\end{equation}
It is known (see \cite{boyd1981}) that the function $\ba \mapsto \nu(\ba)$ is unbounded, and a stronger conclusion follows
from our Lemma \ref{lemorder2}.  Moreover, if $\nu(\ba)$ is sufficiently large, then the map $\bk \mapsto \bk^T \ba$ 
restricted to points $\bk$ in the finite subset $\sS$ takes distinct integer values, and therefore induces an ordering in 
$\sS$.  This follows immediately from the triangle inequality for the norm (\ref{order7}), and was noted in \cite{boyd1981}.  
As this result will be important in our proof of Theorem \ref{thmintro2}, we prove it here as a separate lemma.

\begin{lemma}{\sc [D.~Boyd]}\label{lemorder1}  Let $\sS \subseteq \Z^M$ be a nonempty, finite subset with cardinality
$|\sS| = N + 1$, and let $\ba \not= \bo$ be a point in $\Z^M$ such that
\begin{equation}\label{order28}
2 \|\sS\|_{\infty} < \nu(\ba).
\end{equation}
Then
\begin{equation}\label{order30}
\big\{\bk^T \ba : \bk \in \sS\big\}
\end{equation}
is a collection of $N + 1$ distinct integers.
\end{lemma}

\begin{proof}  If $N = 0$ the result is trivial.  Assume that $1 \le N$, and let $\bk$ and $\bell$ be distinct points
in $\sS$.  If
\begin{equation*}\label{order35}
\bk^T \ba = \bell^T \ba,
\end{equation*} 
then
\begin{equation*}\label{order40}
(\bk - \bell)^T \ba = \bo.
\end{equation*}
It follows that
\begin{equation*}\label{order45}
\nu(\ba) \le \|\bk - \bell\|_{\infty} \le \|\bk\|_{\infty} + \|\bell\|_{\infty} \le 2 \|\sS\|_{\infty},
\end{equation*}
and this contradicts the hypothesis (\ref{order28}).  We conclude that (\ref{order30}) contains $N + 1$ distinct integers.
\end{proof}

Let $\vphi_{\balpha} : \Z^M \rightarrow \R$ be an injective homomorphism, and let $\sS \subseteq \Z^M$ be a nonempty,
finite subset of cardinality $N + 1$.  We assume that the elements of $\sS$ are indexed so that both (\ref{order1}) 
and (\ref{order5}) hold.  If $\ba \not= \bo$ in $\Z^M$ satisfies (\ref{order28}), then it may happen that the indexing 
(\ref{order1}) also satisfies the system of inequalities
\begin{equation*}\label{order50}
\bk_0^T \ba < \bk_1^T \ba < \bk_2^T \ba < \cdots < \bk_N^T \ba.
\end{equation*}
We write $\B(\balpha, \sS)$ for the collection of such lattice points $\ba$.  That is, we define
\begin{equation}\label{order55}
\begin{split}
\B(\balpha, \sS) &= \big\{\ba \in \Z^M : \text{$2 \|\sS\|_{\infty} < \nu(\ba)$}\\
	&\qquad\qquad\text{and $\bk_0^T \ba < \bk_1^T \ba < \bk_2^T \ba < \cdots < \bk_N^T \ba$}\big\}.
\end{split}
\end{equation}
The following lemma establishes a crucial property of $\B(\balpha, \sS)$.

\begin{lemma}\label{lemorder2}  Let the subset $\B(\balpha, \sS)$ be defined by {\rm (\ref{order55})}.  Then $\B(\balpha, \sS)$ 
is an infinite set, and the function $\nu$ restricted to $\B(\balpha, \sS)$, is unbounded on $\B(\balpha, \sS)$.
\end{lemma}

\begin{proof}  By hypothesis
\begin{equation}\label{order265}
\eta = \eta(\balpha, \sS) = \min\big\{\bk_n^T \balpha - \bk_{n-1}^T \balpha : 1 \le n \le N\big\}
\end{equation}
is a positive constant that depends on $\balpha$ and $\sS$.  

By Dirichlet's theorem in Diophantine approximation (see \cite{cassels1965} or \cite{schmidt1980}), for each 
positive integer $Q$ there exists an integer $q$ such that $1 \le q \le Q$, and
\begin{equation}\label{order270}
\max\big\{\|q \alpha_m\| : m = 1, 2, \dots , M\big\} \le (Q + 1)^{-\frac{1}{M}} \le (q + 1)^{-\frac{1}{M}},
\end{equation}
where $\|\ \|$ on the left of (\ref{order270}) is the distance to the nearest integer function.  Let $\qq$ be the collection of
positive integers $q$ such that
\begin{equation}\label{order272}
\max\big\{\|q \alpha_m\| : m = 1, 2, \dots , M\big\} \le (q + 1)^{-\frac{1}{M}}.
\end{equation}
Because $2 \le M$, at least one of the coordinates $\alpha_m$ is irrational, and it follows from (\ref{order270}) that
$\qq$ is an infinite set.  

For each positive integer $q$ in $\qq$, we select integers $b_{1 q}, b_{2 q}, \dots , b_{M q}$, so that
\begin{equation}\label{order275}
\|q \alpha_m\| = |q \alpha_m - b_{m q}|,\quad\text{for $m = 1, 2, \dots , M$}.
\end{equation}
Then (\ref{order272}) can be written as
\begin{equation}\label{order277}
\max\big\{|q \alpha_m - b_{m q}| : m = 1, 2, \dots , M\big\} \le (q + 1)^{-\frac{1}{M}}.
\end{equation}
Let $\bb_q  = \bigl(b_{m q}\bigr)$ be the corresponding lattice point in $\Z^M$, so that $q \mapsto \bb_q$ is a map from 
$\qq$ into $\Z^M$.  It follows using (\ref{order265}) and (\ref{order277}), that for each index $n$ we have
\begin{equation*}\label{order295}
\begin{split}
q \eta &\le q \bk_n^T \balpha - q \bk_{n-1}^T \balpha\\
          &= \bk_n^T \bb_q - \bk_{n-1}^T \bb_q + \bigl(\bk_n - \bk_{n-1}\bigr)^T (q \balpha - \bb_q)\\
          &\le \bk_n^T \bb_q - \bk_{n-1}^T \bb_q + 2 \|\sS\|_{\infty} \biggl(\sum_{m = 1}^M |q\alpha_m - b_{m q}|\biggr)\\
          &\le \bk_n^T \bb_q - \bk_{n-1}^T \bb_q + 2 \|\sS\|_{\infty} M (q + 1)^{-\frac{1}{M}}.
\end{split}
\end{equation*}
Therefore for each sufficiently large integer $q$ in $\qq$, the lattice point $\bb_q$ satisfies the system of inequalities
\begin{equation*}\label{order296}
\bk_0^T \bb_q < \bk_1^T \bb_q < \bk_2^T \bb_q < \cdots < \bk_N^T \bb_q.
\end{equation*}
We conclude that for a sufficiently large integer $L$ we have
\begin{equation}\label{order298}
\big\{\bb_q : \text{$L \le q$ and $q \in \qq$}\big\} \subseteq \B(\balpha, \sS).
\end{equation}
This shows that $\B(\balpha, \sS)$ is an infinite set.

To complete the proof we will show that the function $\nu$ is unbounded on the infinite collection of lattice points
\begin{equation}\label{order302}
\big\{\bb_q : \text{$L \le q$ and $q \in \qq$}\big\}.
\end{equation}
If $\nu$ is bounded on (\ref{order302}), then there exists a positive integer $B$ such that
\begin{equation}\label{order305}
\nu(\bb_q) \le B
\end{equation}
for all points $\bb_q$ in the set (\ref{order302}).  Let $\cc_B$ be the finite set
\begin{equation*}\label{order310}
\cc_B = \big\{\bc \in \Z^M : 1 \le \|\bc\|_{\infty} \le B\big\}.
\end{equation*}
Because $\alpha_1, \alpha_2, \dots , \alpha_M$, are $\Q$-linearly independent, and $\cc_B$ is a finite set of nonzero
lattice points, we have
\begin{equation}\label{order320}
0 < \delta_B = \min\bigg\{\biggl|\sum_{m = 1}^M c_m \alpha_m\biggr| : \bc \in \cc_B\bigg\}.
\end{equation}
By our assumption (\ref{order305}),  for each point $\bb_q$ in (\ref{order302}) there exists a point 
$\bc_q = (c_{m q})$ in $\cc_B$, such that 
\begin{equation}\label{order325}
\bc_q^T \bb_q = \sum_{m = 1}^M c_{m q} b_{m q} = 0.
\end{equation}
Using (\ref{order277}) and (\ref{order325}), we find that
\begin{equation}\label{order330}
\begin{split}
q \delta_B &\le q \biggl|\sum_{m = 1}^M c_{m q} \alpha_m \biggr|\\
             &= \biggl|\sum_{m = 1}^M c_{m q}\bigl(q \alpha_m - b_{m q}\bigr)\biggr|\\
             &\le \biggl(\sum_{m = 1}^M |c_{m q}|\biggr) \max\big\{|q \alpha_m - b_{m q}| : m = 1, 2, \dots , M\big\}\\
             &\le M B (q + 1)^{-\frac{1}{M}}.
\end{split}
\end{equation}
But (\ref{order330}) is impossible when $q$ is sufficiently large, and the contradiction implies that the assumption
(\ref{order305}) is false.  We have shown that $\nu$ is unbounded on the set (\ref{order302}).  In view of (\ref{order298}),
the function $\nu$ is unbounded on $\B(\balpha, \sS)$.
\end{proof}

\section{Proof of Theorem \ref{thmintro2}}  

If $M = 1$ then the inequality (\ref{intro125}) follows from Corollary \ref{corintro1}.  Therefore we assume that $2 \le M$.

Let $\vphi_{\balpha} : \Z^M \rightarrow \R$ be an injective homomorphism, and let the set $\sS$ be indexed 
so that {\rm (\ref{intro115})} and {\rm (\ref{intro120})} hold.  It follows from Lemma \ref{lemorder2} that the collection of
lattice points $\B(\balpha, \sS)$ defined by (\ref{order55}), is an infinite set, and the function $\nu$ defined by 
(\ref{order25}) is unbounded on $\B(\balpha, \sS)$.

Let $\ba$ be a lattice point in $\B(\balpha, \sS)$.  If $F : (\R/\Z)^M \rightarrow \C$ is given by (\ref{intro70}), we  define
an associated trigonometric polynomial $F_{\ba} : \R/\Z \rightarrow \C$ in one variable by
\begin{equation}\label{order360}
F_{\ba}(t) = \sum_{\bk \in \sS} \tF(\bk) e\bigl(\bk^T \ba t\bigr) = \sum_{n = 0}^N \tF(\bk_n) e\bigl(\bk_n^T \ba t\bigr),
\end{equation}
where the equality on the right of (\ref{order360}) uses the indexing (\ref{intro115}) induced by $\vphi_{\balpha}$.  The 
hypothesis (\ref{intro120}) implies that the integer exponents on the right of (\ref{order360}) satisfy the system of inequalities
\begin{equation}\label{order365}
\bk_0^T \ba < \bk_1^T \ba < \bk_2^T \ba < \cdots < \bk_N^T \ba.
\end{equation}
Then it follows from (\ref{intro40}), (\ref{order360}), and (\ref{order365}), that
\begin{equation}\label{order370}
\bigl|\tF\bigl(\bk_n\bigr)\bigr| \le \binom{N}{n} \mM(F_{\ba}),\quad\text{for each $n = 0, 1, 2, \dots , N$}.
\end{equation}
We have proved that the system of inequalities (\ref{order370}) holds for each lattice point $\ba$ in $\B(\balpha, \sS)$. 

To complete the proof we appeal to an inequality of Boyd \cite[Lemma 2]{boyd1981}, which asserts that if $\bb$ is a
parameter in $\Z^M$ then
\begin{equation}\label{order375}
\limsup_{\nu(\bb) \rightarrow \infty} \mM\bigl(F_{\bb}\bigr) \le \mM(F).
\end{equation}
More precisely, if $\bb_1, \bb_2, \bb_3, \dots $, is a sequence of points in $\Z^M$ such that 
\begin{equation}\label{order380}
\lim_{j \rightarrow \infty} \nu(\bb_j) = \infty, 
\end{equation}
then
\begin{equation}\label{order385}
\limsup_{j \rightarrow \infty} \mM\bigl(F_{\bb_j}\bigr) \le \mM(F).
\end{equation}
Because $\nu$ is unbounded on $\B(\balpha, \sS)$, there exists a sequence $\bb_1, \bb_2, \bb_3, \dots $, 
contained in $\B(\balpha, \sS)$ that satisfies (\ref{order380}).  Hence the sequence $\bb_1, \bb_2, \bb_3, \dots $,
in $\B(\balpha, \sS)$ also satisfies (\ref{order385}).  From (\ref{order370}) we have
\begin{equation}\label{order390}
\bigl|\tF\bigl(\bk_n\bigr)\bigr| \le \binom{N}{n} \mM(F_{\bb_j}),
\end{equation}
for each $n = 0, 1, 2, \dots , N$, and for each $j = 1, 2, 3, \dots $.  The inequality (\ref{intro125}) plainly follows from 
(\ref{order385}) and (\ref{order390}).  This completes the proof of Theorem \ref{thmintro2}.
\medskip

Boyd conjectured in \cite{boyd1981} that (\ref{order375}) could be improved to
\begin{equation}\label{order400}
\lim_{\nu(\bb) \rightarrow \infty} \mM\bigl(F_{\bb}\bigr) = \mM(F).
\end{equation}
The proposed identity (\ref{order400}) was later verified by Lawton \cite{lawton1983} (see also \cite{dobrowolski2017} and
\cite{lalin2013}).  Here we have used Boyd's inequality (\ref{order375}) because it is simpler to prove than (\ref{order400}), 
and the more precise result (\ref{order400}) does not effect the inequality (\ref{intro125}).

\end{document}